\documentclass[11pt]{article}
\usepackage{amssymb, amsfonts, amsmath}
\usepackage{graphicx}
\usepackage{subfigure}
\usepackage{float}
\usepackage{color}
\usepackage{slashbox}
\usepackage{amsthm}
\newcommand{\bb}{\begin{eqnarray*}}
\newcommand{\ee}{\end{eqnarray*}}
\newcommand{\bbl}{\begin{eqnarray}}
\newcommand{\eel}{\end{eqnarray}}
\newcommand{\LB}{\begin{Lemma}}

\newcommand{\LE}{\end{Lemma}}
\newcommand{\CB}{\begin{Corollary}}
\newcommand{\CE}{\end{Corollary}}
\newcommand{\PB}{\begin{Proposition}}
\newcommand{\PE}{\end{Proposition}}
\newcommand{\TB}{\begin{Theorem}}
\newcommand{\TE}{\end{Theorem}}

\catcode`@=11
\@addtoreset{equation}{section}
\catcode`@=12
\newtheorem{Lemma}{Lemma}[section]

\newtheorem{Proposition}[Lemma]{Proposition}

\newtheorem{Theorem}[Lemma]{Theorem}
\newtheorem{Corollary}[Lemma]{Corollary}

\begin{document}
\begin{center}
{\large \bf Equilibrium distributions and discrete Schur-constant models}\\

\vspace{7mm}
{Anna Casta\~ner, M. Merc\`e Claramunt}\\
{Universitat de Barcelona, Departament de Matem\`atica Econ\`omica, Financera i Actuarial,\\
690 Avinguda Diagonal, E-08034 Barcelona, Spain}
\end{center}

%%%%%%%%%%%%%%%%%%%%%%%%%%%%%%%%%%%%%%%%%%%%%%%%%%%%%%%%%%%%%%%%%%
\vspace{3mm}
\begin{abstract}
This paper introduces Schur-constant equilibrium distribution models of dimension $n$ for arithmetic non-negative random variables. Such a model is defined through the (several orders) equilibrium distributions of a univariate survival function. First, the bivariate case is considered and analyzed in depth, stressing the main characteristics of the Poisson case. The analysis is then extended to the multivariate case. Several properties are derived, including the implicit correlation and the distribution of the sum.
 
\end{abstract}

\noindent \textbf{Keywords}: Schur-constant property;  discrete stationary-excess operator; discrete equilibrium distributions.

\noindent \textbf{Mathematics Subject Classification (2010)}: 60E05, 62H05. 

%%%%%%%%%%%%%%%%%%%%%%%%%%%%%%%%%%%%%%%%%%%%%%%%%%%%%%%%%%%%%%%%%%
\section{Introduction}
%%%%%%%%%%%%%%%%%%%%%%%%%%%%%%%%%%%%%%%%%%%%%%%%%%%%%%%%%%%%%%%%%%
Schur-constant models for discrete survival data have been studied by several authors including Casta\~ner et al. (2015), Lef\`evre et al. (2017) and Ta and Van (2017). In this case the discrete setting means that it is valued at $\mathbb{N}_0=\left\{0,1,...\right\}$. There is also a vaste literature regarding Schur-constant models for continuous survival data valued in $\mathbb{R}_+$; let us mention, among others, Caramellino and Spizzichino (1994), Chi et al. (2009) and Nelsen (2005). Casta\~ner et al. (2015) discuss the properties of  a Schur-constant vector $(X_1, \ldots, X_n)$ generated by the survival function of an univariate random variable (r.v.). In that paper, it is shown that this survival function has to be $n$-monotone.  The present paper aims to develop a class of Schur-constant models generated by an admissible univariate survival function. The admissibility condition is the existence of the $(n-1)-th$ equilibrium distribution of the univariate r.v. Because of it, this new model is called \emph{Schur-constant (multivariate) equilibrium distribution model}. Some properties of the bivariate continuous case have been previously studied by Nair and Sankaran (2014).  For the sake of completeness, we include in this introduction some definitions and well-known results.

Let $(X_1, \ldots, X_n)$ be a vector of $n$ ($\geq 2$) arithmetic non-negative random variables, called lifetimes. It is said to have a Schur-constant joint survival function if for all $(x_1, \ldots, x_n) \in \mathbb{N}_0^n$,

\begin{equation}\label{Intro.1}
P(X_1\geq x_1, \ldots, X_n\geq x_n) = S(x_1+\ldots +x_n),
\end{equation}
where $S$ is an admissible function from $\mathbb{N}_0$ to $[0,1]$. In fact, $S$ is the survival function of each of the marginal r.v. $X_i$ and it has to be $n$-monotone on $\mathbb{N}_0$. A function $f(x)$: $\mathbb{N}_0 \to \mathbb{R}$ is said to be $n$-monotone if it satisfies
\begin{equation}\label{Intro.2}
(-1)^j\, \Delta^j f(x) \geq 0, \text{ for } j=0, \ldots, n,
\end{equation}
where $\Delta$ is the forward difference operator (i.e. $\Delta f(x)= f(x+1)-f(x)$) and $\Delta^j$ is its $j-th$ iterated. 

A general representation (Casta\~ner et al., 2015) valid for any discrete Schur-constant model, putting ${a \choose b}= 0$ when $a<b$, is: 
\begin{equation}\label{1.1}
S(x_1+\ldots+x_n) = E\left[{Z-(x_1+\ldots+x_n)+n-1\choose n-1}/{Z+n-1\choose n-1}\right],
\end{equation}
where the variable $Z$ is distributed as $X_1+\ldots + X_n$, i.e., with a probability mass function (p.m.f.) given by 
\begin{equation}\label{1.2}
P(Z=z) = (-1)^n\, \Delta^n S(z)\, {z+n-1\choose n-1}.
\end{equation}
As a particular case, for  $n=2$, 
\begin{equation}\label{1.3}
P(Z=z)=\Delta^2 S(z)(z+1).
\end{equation}

The p.m.f. of any subvector in $(X_1, \ldots, X_n)$ can be obtained from the $j-th$ iterated forward difference of $S$,
\begin{equation}\label{1.a}
P(X_1=x_1, \ldots, X_j=x_j) = (-1)^j\, \Delta^jS(x_{1}+ \ldots+ x_{j}). 
\end{equation}

For bivariate Schur-constant vectors we can obtain an easy expression of the p.m.f. of the marginal. This is the next Lemma \ref{l.1}, which will be used in the next section of the paper.

\begin{Lemma}\label{l.1}
If the vector $(X_1,X_2)$ is Schur-constant, the probability mass function of $X_1$ fulfills

\begin{equation}\label{Z.1}
P(X_1=x+1)=P(X_1=x) - \frac{P(Z=x)}{x+1},
\end{equation}
for $x\geq 0$, with $$P(X_1=0)=E\left(\frac{1}{Z+1}\right).$$
\end{Lemma}

\begin{proof}
The survival function of $X_1$, for $n=2$, is obtained from (\ref{1.1}) when $x_2=0$,
\begin{equation}\label{Z.2}
S(x)=E\left[\frac{(Z-x+1)_+}{Z+1}\right]= \sum_{z=x}^{\infty} \frac{z-x+1}{z+1}P(Z=z).
\end{equation}
From definition, $P(X_1=x)=-\Delta S(x)$. Then, considering (\ref{Z.2}),
\begin{equation}\label{Z.0}
P(X_1=x)= \sum_{z=x}^{\infty} \frac{1}{z+1} P(Z=z).
\end{equation}
From (\ref{Z.0}), for $x=0$, $P(X_1=0)=E\left(\frac{1}{Z+1}\right)$. Now, let us apply again the definition of the survival function so that $\Delta^2 S(x)=P(X_1=x)-P(X_1=x+1)$. Using this expression and (\ref{1.3}), the desired result (\ref{Z.1}) is obtained.
\end{proof}
On the basis of these definitions, we develop our analysis. The paper is organized as follows. In Section 2, we introduce the discrete Schur-constant bivariate equilibrium distribution model and provide its characteristics. We prove that in the Poisson bivariate Schur-constant equilibrium distribution model, $Z$ is also Poisson distributed. In Section 3, we extend the analysis to the multivariate case. The paper ends with some concluding remarks.

%%%%%%%%%%%%%%%%%%%%%%%%%%%%%%%%%%%%%%%%%%%%%%%%%%%%%%%%%%%%%%%%%%
\section{Schur-constant bivariate equilibrium distribution model}
%%%%%%%%%%%%%%%%%%%%%%%%%%%%%%%%%%%%%%%%%%%%%%%%%%%%%%%%%%%%%%%%%%
In renewal theory, the asymptotic distribution of the age or residual life in a renewal process is known as the univariate equilibrium distribution (Cox, 1962).  Let $X$ be a non negative discrete random variable representing the lifetime of a component with finite mean $\mu$ and survival function $S(x)=P(X\geq x), x \geq 0$. The (first order) equilibrium distribution is defined through the probability mass function, as

\begin{equation}\label{2*.6}
p^*(x)=\frac{S(x+1)}{\mu}, \text{ for }x\in \mathbb{N}_0,
\end{equation}
or, via the survival function, as
\begin{equation}\label{2*.7}
S^*(x)=\sum_{h=x+1}^\infty \frac{S(h)}{\mu}, \text{ for }x\in \mathbb{N}_0.
\end{equation}
Expression (\ref{2*.7}) is well defined as survival function as long as it is a decreasing function and $S^*(0)=1$. This latter condition can be easily checked taking into account that $E(X)= \sum_{h=1}^\infty S(h)$.
 
Much attention has been paid to the equilibrium distributions (of first and higher orders) associated with a given distribution function, but most studies are for continuous univariate random variables (see for instance Deshpande et al. (1986), Sunoj (2004) and the references therein). Regarding the discrete univariate case, two papers should be mentioned: Li (2011) and Willmot et al. (2005). In addition, Lef\`evre and Loisel (2010) introduce a specific discrete version of the stationary-excess operator for discrete non-negative random variables. This discrete stationary-excess operator $H$ maps any non-negative discrete random variable $X$ to an associated discrete non-negative random variable $X_H$ whose probability mass function is
%\begin{equation}\label{2*.71}
$$
P(X_H =x)=\frac{P(X\geq x+1)}{E(X)},\text{ \ \ } x\in \mathbb{N}_0.
$$
%\end{equation}
Notice that the equilibrium distribution can be also obtained by applying this discrete stationary-excess operator. 

In this section, we discuss a (first order) bivariate equilibrium distribution of discrete random variables. It represents in fact a family of discrete Schur-constant distributions, that we call \emph{Schur-constant bivariate equilibrium distributions}. Although it has been recently analized by Gupta (2012) and Nair and Sankaran (2014) in a continuous setting, as far as we know, the bivariate discrete case has not previously been evaluated. The asymptotic joint distribution of age and residual life is defined by the survival function
\begin{equation}\label{2*.8}
G(x,y)=\frac{1}{\mu}\sum_{h=x+y+1}^\infty S(h)=S^*(x+y).
\end{equation}

Now we prove, in the next proposition, that this survival function corresponds to a Schur-constant survival function.
\begin{Proposition}\label{Prop.2.1}
The random vector $(X^*, Y^*)$ with survival function $G$ given by (\ref{2*.8}) is Schur-constant and its marginal distribution is the equilibrium distribution of $X$ given by (\ref{2*.7}). 
\end{Proposition}
\begin{proof}
It is enough to put $y=0$ in expression (\ref{2*.8}) to check that $X^*$ is the equilibrium distribution of $X$. With respect to the Schur-constancy property, from (\ref{2*.8}), we obtain
\begin{equation}\label{2*.9}
\Delta S^*(x)= -\frac{S(x+1)}{\mu}\leq0,
\end{equation}
and
\begin{eqnarray}\label{2*.10}
\Delta^2 S^*(x) &=& -\frac{S(x+2)}{\mu}+\frac{S(x+1)}{\mu} = \frac{1}{\mu}\left[P(X\geq x+1)-P(X\geq x+2)\right] \nonumber\\
 &=& \frac{P(X=x+1)}{\mu} \geq 0.
\end{eqnarray}
Then, from (\ref{2*.9}) and (\ref{2*.10}), $S^*(x)$ is always $2$-monotone although $S(x)$ may not.
\end{proof}

Now we derive several useful formulas. As $(X^*, Y^*)$ is Schur-constant,
\begin{equation}\label{2*.11}
P(X^*=x, Y^*=y)=\Delta^2S^*(x+y)=\frac{P(X=x+y+1)}{\mu}, 
\end{equation}
and the p.m.f of the sum $Z=X^*+Y^*$, by (\ref{1.3}) and (\ref{2*.10}), is

\begin{equation}\label{2*.11a}
P(Z=z)=\frac{P(X=z+1)(z+1)}{\mu}.
\end{equation}
In fact, $Z$ turns out to be the (first) length-biased type transform of $X$ (as defined in Lef\`evre and Loisel, 2013).

The basic characteristics of $(X^*,Y^*)$ are available in terms of the expectations of simple functions of $X$ and this is our next proposition.

\begin{Proposition}\label{Prop.2.2}
Let $(X^*,Y^*)$ be as in Proposition \ref{Prop.2.1}. Then,
\begin{equation}\label{2*.12}
\mu^*=E(X^*)= \frac{1}{2} \left(\frac{E(X^2)}{\mu} -1\right),
\end{equation}
\begin{equation}\label{2*.13}
V(X^*)=\frac{4\mu E(X^3)-3\left(E(X^2)\right)^2-\mu^2}{12\mu^2}.
\end{equation}
\end{Proposition}
\begin{proof}
Considering (\ref{2*.6}), we have
\begin{equation}\label{2*.14}
\mu^*=E(X^*)= \frac{1}{\mu}\sum_{x=0}^\infty x S(x+1).
\end{equation}
It is a simple exercise that $\Delta \left(\frac{x^2-x}{2}\right)=x$, so that we can apply summation by parts to the last sum in (\ref{2*.14}). This leads to
\begin{eqnarray}\label{2*.15}
 \sum_{x=0}^\infty x S(x+1)&=& \frac{1}{2} \left(\sum_{x=0}^\infty p(x+1) (x+1)^2- \sum_{x=0}^\infty p(x+1) (x+1)\right)\nonumber\\
&=& \frac{1}{2} \left(\sum_{x=0}^\infty p(x) x^2- \sum_{x=0}^\infty p(x) x\right).
\end{eqnarray}
Substituting this in (\ref{2*.14}), we obtain (\ref{2*.12}). Let us derive now a formula for the second ordinary moment of $X^*$,
\begin{equation}\label{2*.16}
E((X^*)^2)= \sum_{x=0}^\infty x^2 \frac{S(x+1)}{\mu}.
\end{equation}
We know that $\Delta \left(\frac{x^3}{3}-\frac{x^2}{2}+\frac{x}{6}\right)=x^2$. Considering this, summation by parts yields

\begin{equation}\label{2*.17}
E((X^*)^2)= \frac{1}{\mu}\left(\frac{E(X^3)}{3}-\frac{E(X^2)}{2}+\frac{\mu}{6}\right).
\end{equation}
By (\ref{2*.12}) and (\ref{2*.17}), we obtain formula (\ref{2*.13}).
\end{proof}

Casta\~ner et al. (2015) presented simple formulas for the Pearson correlation coefficient, $\rho$, between any two variables in a Schur-constant vector. One of them allows to compute $\rho$ in terms of the expectation and the variance of one of the marginal random variables. But in these Schur-constant discrete bivariate equilibrium distributions, it is convenient to relate $\rho$ to the main characteristics of the original variable $X$, the building block of the model.

\begin{Proposition}\label{Prop.2.3}
In terms of $E(X^*)$ and $V(X^*)$,
\begin{equation} \label{2*.171}
\rho= \frac{V(X^*)-\left(E(X^*)\right)^2-E(X^*)}{2V(X^*)},
\end{equation}
and, in terms of the ordinary moments of $X$,
\begin{equation} \label{2*.18}
\rho=\frac{2\mu E(X^3)-3\left(E(X^2)\right)^2+\mu^2}{4\mu E(X^3)-3\left(E(X^2)\right)^2-\mu^2}.
\end{equation}

\end{Proposition}
\begin{proof} Formula (\ref{2*.171}) is direct from Proposition 6.2. in Casta\~ner et al. (2015). Let us now establish (\ref{2*.18}). We could evaluate $E(X^*Y^*)$ directly using (\ref{2*.11}). However, a simpler method consists in using  (\ref{2*.171}) in which $E(X^*)$  is substituted by (\ref{2*.12}) and $V(X^*)$  is substituted by (\ref{2*.13}). Simple algebra yields the desired formula.
\end{proof}

The corresponding formulas (\ref{2*.12}), (\ref{2*.13}) and (\ref{2*.18}) for continuous random variables can be found in Nair and Sankaran (2014).

Now, we apply these results to two specific models generated by  well-known distributions, the Poisson one and the geometric one.
%%%%%%%%%%%%%%%%%%%%%%%%%%%%%%%%%%%%%%%%%%%%%%%%%%%%%%%%%%%%%%%%%%

\subsection{Poisson Schur-constant bivariate equilibrium model}
%%%%%%%%%%%%%%%%%%%%%%%%%%%%%%%%%%%%%%%%%%%%%%%%%%%%%%%%%%%%%%%%%%
Let us consider that the  Schur-constant equilibrium distribution model $(X_1,X_2)$ is generated by a r.v. $X$ that follows a Poisson distribution with parameter $\lambda$; i.e., $P(X=k)=e^{-\lambda}\frac{\lambda^{k}}{k!}$. This model is an interesting and particular one, because of its properties. Using (\ref{2*.11a}), it is immediate to conclude that $Z=X_1+X_2$ also follows a Poisson distribution of parameter $\lambda$, 
\begin{equation} \label{2*.19}
P(Z=z)=\frac{\lambda^{z+1}e^{-\lambda}(z+1)}{\lambda (z+1)!}=\frac{\lambda^{z}e^{-\lambda}}{z!},\text{ \ \ }    z=0,1,\ldots.
\end{equation}

\begin{Proposition}\label{prop:PZPoisson}
Let  $(X_1,X_2)$ be a Schur-constant equilibrium distribution model generated by a r.v. $X$ and $Z=X_1+X_2$. Then, $Z=_d X$ if and only if $X$ follows a Poisson distribution.
\end{Proposition}
\begin{proof}  It is well known from Panjer (1981) that the Poisson distribution is the only one that fulfills a recursive equation on its probabilities of the form
\begin{equation} \label{2*.191}
P(X=k+1)=P(X=k) \frac{b}{k+1},\text{ \ \ } k=0,1,2,\ldots ,  
\end{equation}
being $b=\lambda$. Then, substituting (\ref{2*.191}) in (\ref{2*.11a}), we obtain
\begin{equation} \label{2*.20}
P(Z=z)=P(X=z),\text{ \ \ }    z=0,1,2,\ldots. 
\end{equation}
For the other implication, if $Z$ is Poisson distributed with parameter $\lambda$, from (\ref{2*.11a}), $P(X=z+1)=\frac{P(Z=z)\lambda}{z+1}$, which implies that $X$ is Poisson distributed with the same parameter. 
\end{proof}

The distribution of $X_1$ can be found using the property that $X_1$ is the equilibrium distribution of $X$. Since $X$ is Poisson distributed, $S(x)= \sum_{h=x}^\infty \frac{e^{-\lambda}\lambda^h}{h!}$ and then

\begin{equation} \label{2*.21}
P(X_1=x)=\frac{1-\frac{\Gamma(x+1,\lambda)}{x!}}{\lambda} , \text{ \ \ }  x=0, 1, \ldots,
\end{equation}
being $\Gamma(\cdot\,,\cdot)$ the incomplete Gamma function.  The same result can be obtained using (\ref{Z.1}) and Proposition  \ref{prop:PZPoisson}. The correlation coefficient between $X_1$ and $X_2$ in this model, computed using (\ref{2*.18}), is $\rho=\frac{-\lambda}{6+\lambda}$.
%%%%%%%%%%%%%%%%%%%%%%%%%%%%%%%%%%%%%%%%%%%%%%%%%%%%%%%%%%%%%%%%%%

\subsection{Geometric Schur-constant bivariate equilibrium model}
%%%%%%%%%%%%%%%%%%%%%%%%%%%%%%%%%%%%%%%%%%%%%%%%%%%%%%%%%%%%%%%%%%
Let us consider a geometric r.v. $X$ with survival function $S(x)=q^x$. From this survival function we build a Schur-constant equilibrium vector $(X_1, X_2)$. The r.v. $X_1$ is the equilibrium distribution of $X$. It immediately follows that $S^*(x)=S(x)=q^x$. Then, $X_1$ and $X_2$ are independent r.v. (Casta\~ner et al., 2015). So, the Pearson correlation coefficient $\rho$ computed from (\ref{2*.171}) is equal to $0$. 

%%%%%%%%%%%%%%%%%%%%%%%%%%%%%%%%%%%%%%%%%%%%%%%%%%%%%%%%%%%%%%%%%%
\section{Schur-constant multivariate equilibrium distribution model}
%%%%%%%%%%%%%%%%%%%%%%%%%%%%%%%%%%%%%%%%%%%%%%%%%%%%%%%%%%%%%%%%%%
In this section we generalize the results of the previous section. The first order equilibrium distribution of a discrete random variable $X$ has been used in Section 2 to define a special class of bivariate Schur-constant models. In the present section, this analysis is extended to the general multivariate case through the use of the $n-th$ order equilibrium distribution.
 
Let $X$ be a non negative discrete random variable with finite mean $\mu$ and survival function $S(x)$. The second order equilibrium distribution can be defined  by its survival function 
\begin{equation}\label{4.1}
S^{2*}(x)=\sum_{h=x+1}^\infty \frac{S^*(h)}{\mu_{1:1}}, \text{ for }x\in \mathbb{N}_0,
\end{equation}
or, alternatively, by its p.m.f 
\begin{equation}\label{4.2}
p^{2*}(x)=\frac{S^*(x+1)}{\mu_{1:1}}, \text{ for }x\in \mathbb{N}_0,
\end{equation}
being $\mu_{1:1}$ the first order moment of $X^*$ (the first order equilibrium distribution of $X$). Likewise, we can define recursively the $n-th$ order equilibrium distribution,
\begin{equation}\label{4.3}
S^{n*}(x)=\sum_{h=x+1}^\infty \frac{S^{(n-1)*}(h)}{\mu_{n-1:1} }, \text{ for }x\in \mathbb{N}_0,
\end{equation}
and
\begin{equation}\label{4.4}
p^{n*}(x)=\frac{S^{(n-1)*}(x+1)}{\mu_{n-1:1}}, \text{ for }x\in \mathbb{N}_0,
\end{equation}
being $\mu_{i:1}<\infty$ the first order moment of the $i-th$ order equilibrium distribution of $X$, with $\mu_{0:1}=\mu$, and $S^{(0)*}(x)=S(x)$.

Following the same procedure than in Section 2, lets us define the random vector $(X_1, \ldots, X_n)$ with survival function 
\begin{equation}\label{4.4a}
P(X_1 \geq x_1, \ldots, X_n \geq x_n)= S^{(n-1)*}(x_1+ \ldots + x_n).
\end{equation}

Our next proposition is the parallel of Proposition \ref{Prop.2.1} in the bivariate case. Now, we prove that the survival function (\ref{4.4a}) corresponds to a Schur-constant survival function.
\begin{Proposition}\label{Prop.A}
The vector $(X_1, \ldots, X_n)$ with survival function  (\ref{4.4a}) is Schur-constant and the marginal $X_1$ is the $(n-1)-th$ order equilibrium distribution of $X$ ((\ref{4.3}) and (\ref{4.4})).
\end{Proposition}

\begin{proof}
We know from Proposition \ref{Prop.2.1} that it is true for $n=2$, because $S^*(x)$ is $2$-monotone. We assume now that it is true for $n-1$, i.e., $S^{(n-2)*}$ is $(n-1)$-monotone. From (\ref{4.3}),
\begin{equation}\label{4.4b}
(-1)^{n} \Delta^{n}S^{(n-1)*}(x)=(-1)^{n} \Delta^{(n-1)} \Delta S^{(n-1)*}(x)= (-1)^{(n-1)}\Delta^{(n-1)} \frac{S^{(n-2)*}(x+1)}{\mu_{n-2:1}}.
\end{equation}
Since $S^{(n-2)*}$ is $(n-1)$-monotone, expression (\ref{4.4b}) is equal or greater than zero, which in turn implies that  $S^{(n-1)*}$ is $n$-monotone. Putting $x_2=\ldots = x_n =0$ in (\ref{4.4a}), $P(X_1 \geq x_1)= S^{(n-1)*}(x_1)$, and thus, the marginal distribution $X_1$ is the $(n-1)-th$ order equilibrium distribution of $X$. 
\end{proof}

A Schur-constant vector of $n$ components can be defined using a survival function not necessarily $n$-monotone, using the previous proposition; we say then that such vector follows a \emph{Schur-constant multivariate equilibrium distribution model}. By exchangeability, all of the $X_i$'s have the same ordinary moments, and the same Pearson correlation coefficient. In this model we are able to obtain all of them from the ordinary moments of the original distribution $X$, as long as $X_i$ is the $(n-1)-th$ order equilibrium distribution of $X$. Let us first establish the following lemma.

\begin{Lemma}\label{l.2}
The set of discrete functions $F(x)$ such that $\Delta F(x)=x^{n-1}$ is given by $P_n(x)+C$, being $C$ a constant and $P_n(x)$ a polynomial of degree $n$, such that $P_n(x)=\sum_{r=0}^n a_r(n) x^r$ with $a_0(n)=0$ and $a_r(n)$, for $r=1,\ldots,n$, that can be computed recursively,
\begin{eqnarray}\label{4.5}
a_n(n) &=& \frac{1}{n}, \nonumber\\
a_{n-h+1}(n)&=& \frac{-1}{n-h+1}\sum_{r=n-h+2}^n {r\choose n-h} a_r(n), \text{ \ \ } h=2, \ldots , n.
\end{eqnarray}

\end{Lemma}
\begin{proof}
Let us consider $F(x)=P_n(x)+C$ being $P_n(x)= \sum_{r=0}^n a_r(n) x^r$ whit $a_0(n)=0$, a polynomial of degree $n$ on $x$. The difference of any of these functions is exactly the same and equals
\begin{eqnarray}\label{4.6}
\Delta P_n(x) &=& \sum_{r=0}^n a_r(n) \Delta x^r = \sum_{r=0}^n a_r(n) \left(\sum_{s=0}^r {r \choose s}x^s -x^r\right) \nonumber\\
&=& \sum_{r=0}^{n-1} a_r(n) \sum_{s=0}^r  {r+1 \choose s} x^s = \sum_{s=0}^{n-1} x^s \sum_{r=s}^{n-1}  {r+1 \choose s} a_{r+1}(n).
\end{eqnarray}
From (\ref{4.6}), we observe that $\Delta P_n(x)$ is a new polynomial of degree $n-1$ and coefficients $b_s$, that can be derived from the coefficients of the original polynomial, $a_r(n)$,
\begin{equation}\label{4.7}
b_s= \sum _{r=s+1}^n {r \choose s} a_r(n), \text{ \ \ } s=0, \ldots, n-1.
\end{equation}
From (\ref{4.6}) and (\ref{4.7}) and imposing the condition that $\Delta P_n(x)=x^{n-1}$, that is $b_{n-1}=1$, and $b_s=0$ for $s=0,\ldots, n-2$, a linear system of $n$ equations on $a_r(n)$ is defined. The value of $a_n=\frac{1}{n}$ is obtained when $s=n-1$. When $s=n-h$, the corresponding equation is 
\begin{equation}\label{4.8}
0=\sum _{r=n-h+1}^n {r \choose n-h} a_r(n).
\end{equation}
 Isolating $a_{n-h+1}(n)$ from (\ref{4.8}) the statement of the lemma is obtained. 
\end{proof}

Table \ref{tab:Table 3} below gives the values of $a_r(n)$, for the first values of $n$. 
\begin{table}[H]
\caption{Some coefficients $a_r(n)$}
\label{tab:Table 3} \centering
\vspace{3mm}
\begin{tabular}{c|rrrrrrrrr}
\hline
\backslashbox{$r$}{$n$} & $2$ & $3$ & $4$ & $5$ & $6$ & $7$ & $8$ & $9$ & $10$\\ \hline  
$1$ & ${-1}/{2}$ & ${1}/{6}$ & $0$ & ${-1}/{30}$ & $0$ & ${1}/{42}$ & $0$ & ${-1}/{30}$ & $0$ \\  
$2$ & ${1}/{2}$ & ${-1}/{2}$ & ${1}/{4}$ & $0$ & ${-1}/{12}$ & $0$ & ${1}/{12}$ & $0$ & ${-3}/{20}$ \\ 
$3$ &  & ${1}/{3}$& ${-1}/{2}$ &${1}/{3}$& $0$ & ${-1}/{6}$ & $0$ & ${2}/{9}$ & $0$ \\ 
$4$ & &  & ${1}/{4}$ &${-1}/{2}$ & ${5}/{12}$ &$0$ & ${-7}/{24}$ & $0$ & ${1}/{2}$ \\ 
$5$ &  & & &${1}/{5}$ & ${-1}/{2}$ & ${1}/{2}$ & $0$ & ${-7}/{15}$ & $0$ \\ 
$6$ &  &  & & & ${1}/{6}$ & ${-1}/{2}$ & ${7}/{12}$ & $0$ & ${-7}/{10}$ \\ 
$7$ & &  & & &  & ${1}/{7}$ & ${-1}/{2}$ & ${2}/{3}$ & $0$ \\ 
$8$ &  &  &  & &  & & ${1}/{8}$ & ${-1}/{2}$ & ${3}/{4}$ \\ 
$9$ &  &  &  & &  & &  & ${1}/{9}$ & ${-1}/{2}$ \\  
$10$ & & & & & &  & & & ${1}/{10}$ \\ 
\hline
\end{tabular}
\end{table}

\begin{Proposition}\label{Prop.B} The $j-th$ ordinary moment of the $i-th$ order equilibrium distribution, $\mu_{i:j}=E\left[\left(X^{i*}\right)^j\right]$, $j=1,2,...$, $i=1, 2,...$, is given by the recurrence formula
\begin{equation}\label{4.9}
\mu_{i:j}= \frac{1}{\mu_{i-1:1}}\sum_{r=1}^{j+1} \mu_{i-1:r}a_r(j+1),
\end{equation}
with $a_r(\cdot)$ as in Lemma \ref{l.2}, and $\mu_{0:r}=E(X^r)$.
\end{Proposition}
\begin{proof} By definition,

\begin{equation}\label{4.10}
\mu_{i:j}= \frac{1}{\mu_{i-1:1}}\sum_{x=0}^{\infty} x^j S^{(i-1)*}(x+1).
\end{equation}
From Lemma \ref{l.2}, $\Delta\left(\sum_{r=1}^{j+1} a_r(j+1) x^r\right)=x^j$, where the coefficients $a_r(\cdot)$ are given by expression (\ref{4.5}). Therefore, the sum in (\ref{4.10}) can be computed by applying summation by parts,
\begin{equation}\label{4.11}
\sum_{x=0}^{\infty} p^{(i-1)*}(x) \sum_{r=1}^{j+1}a_r(j+1) x^r = \sum_{r=1}^{j+1}a_r(j+1) \sum_{x=0}^{\infty} p^{(i-1)*}(x) x^r.
\end{equation}
\end{proof}
In Li (2011) an equivalent expression relating the $n-th$ factorial moments is found.

We focus now on the study of the Pearson correlation coefficient, $\rho$, in this Schur-constant multivariate equilibrium model. We see that $\rho$ is a function of some ordinary moments of $X$.
\begin{Proposition} \label{Prop.C}Consider the vector $(X_1, \ldots, X_n)$ with survival function (\ref{4.4a}). In terms of expectation and variance of the $(n-1)$-th order equilibrium distribution of $X$, the Pearson correlation coefficient is given by
\begin{equation} \label{4.12}
\rho= \frac{V(X^{(n-1)*})-\left(E\left(X^{(n-1)*}\right)\right)^2-E(X^{(n-1)*})}{2V(X^{(n-1)*})}.
\end{equation}

\end{Proposition}
\begin{proof} 
It is direct from Proposition 6.2. in Casta\~ner et al. (2015).\end{proof}

Thanks to (\ref{4.9}), an equivalent expression of $\rho$ as a function of the ordinary moments of $X$ can be deduced. For instance, let us make explicit the formula when $n=3$ and $n=4$. Let $(X_1,X_2,X_3)$ be a Schur-constant multivariate equilibrium distribution model defined by $X$. The Pearson correlation coefficient can be computed as
\begin{equation}\label{4.13}
\rho =\frac{1}{2}-\frac{\left(\mu-E(X^3)\right)\left(2\mu-3E(X^2)+E(X^3)\right)}{2\mu^2+2(E(X^3))^2+3·E(X^2)(E(X^2)-E(X^4))+\mu\left(-3E(X^2)-4E(X^3)+3E(X^4)\right)}   .\nonumber
\end{equation}
If we consider now the Schur-constant multivariate equilibrium distribution model $(X_1,X_2,X_3, X_4)$,
\begin{equation}\label{4.14}
\rho =\frac{1}{2}-\frac{5\left(6\mu-11E(X^2)+6E(X^3)-E(X^4)\right)\left(2\mu-E(X^2)-2E(X^3)+E(X^4)\right)}{2(D-4\mu J )}, \nonumber
\end{equation}
being
$$
D=36\mu^2+65(E(X^2))^2+20(E(X^3))^2-70E(X^2)E(X^4)+5(E(X^4))^2+24E(X^2)E(X^5)8E(X^3)E(X^5),
$$
and
$$
J=21E(X^2)+8E(X^3)\!\!-\!\!15E(X^4)+4E(X^5).
$$

As an example, Tables \ref{tab:Table ro}  and \ref{tab:Table ro2} include the values of the Pearson correlation coefficient in a Schur-constant multivariate equilibrium distribution model built from a Poisson distributed random variable of mean $\lambda$, as a function of $n$, the number of elements of the random vector.

\begin{table}[H]
\caption{$\rho$ in a Poisson Schur-constant multivariate equilibrium model $(X_1, \ldots, X_n)$ }
\label{tab:Table ro} \centering
\vspace{3mm}
\begin{tabular}{c|cccc}
\hline
$n$ & 2 & 3 & 4 & 5\\ \hline
\\
$\rho$ & $\displaystyle\frac{-\lambda}{6+\lambda}$ & $\displaystyle\frac{-\lambda}{12+2\lambda}$ &  $\displaystyle\frac{-\lambda}{20+3\lambda}$ & $\displaystyle\frac{-\lambda}{30+4\lambda}$\\ \\
\hline%
\end{tabular}
\end{table}

\begin{table}[H]
\caption{$\rho$ in a Poisson Schur-constant multivariate equilibrium model $(X_1, \ldots, X_n)$ for different values of $\lambda$ }
\label{tab:Table ro2} \centering
\vspace{3mm}
\begin{tabular}{c|cccccc}
\hline
\backslashbox{$n$}{$\lambda$} & $0.01$ & $0.5$ & $1$ & $5$ & $10$ & $100$\\ \hline
\\

$2$ & $-0.00166$ & $-0.07692$ & $-0.14286$ & $-0.45455$ & $-0.62500$ & $-0.94340$  \\ \\

$3$ & $-0.00083$ & $-0.03846$ & $-0.07143$ & $-0.22727$ & $-0.31250$ & $-0.47170$ \\ \\

$4$ & $-0.00050$ & $-0.02326$ & $-0.04348$ & $-0.14286$ & $-0.20000$ & $-0.31250$ \\ \\

$5$ & $-0.00033$ & $-0.01563$ & $-0.02941$ & $-0.10000$ & $-0.14286$ & $-0.23256$  \\ \\
\hline%
\end{tabular}
\end{table}

One interesting property of these models is that the probability mass function of $Z=X_1+\ldots + X_n$ is easily calculated from the survival function of $X$. This is our next proposition.
\begin{Proposition} \label{Prop.D}Consider the vector $(X_1, \ldots, X_n)$ with survival function (\ref{4.4a}). The p.m.f. of $Z=X_1+ \ldots + X_n$ is
\begin{equation}\label{4.15}
P(Z=z)= \frac{P(X=z+n-1)}{\mu\cdot \mu_{1:1} \cdots \mu_{(n-2):1}} {z+n-1 \choose n-1}.
\end{equation}

\end{Proposition}
\begin{proof}From Proposition \ref{Prop.A}, the survival function of $X_1$ is $S^{(n-1)*}(x)$, and from its definition (\ref{4.3}),
\begin{equation}\label{4.16}
\Delta^n S^{(n-1)*}(x)= (-1)^n \frac{P(X=x+n-1)}{\mu \cdot\mu_{1:1} \cdots \mu_{(n-2):1}}.
\end{equation}
Inserting (\ref{4.16}) in (\ref{1.2}) yields (\ref{4.15}).
\end{proof}

%%%%%%%%%%%%%%%%%%%%%%%%%%%%%%%%%%%%%%%%%%%%%%%%%%%%%%%%%%%%%%%%%%
\section{Conclusions}
%%%%%%%%%%%%%%%%%%%%%%%%%%%%%%%%%%%%%%%%%%%%%%%%%%%%%%%%%%%%%%%%%%
Discrete Schur-constant models (of dimension $n$) can be generated by univariate survival functions that are $n$-monotone. This paper introduces a class of Schur-constant models that are generated by a wider class of univariate survival functions, i.e., those survival functions such that their $(n-1)-th$ equilibrium distribution exists.

%%%%%%%%%%%%%%%%%%%%%%%%%%%%%%%%%%%%%%%%%%%%%%%%%%%%%%%%%%%%%%%%%%
\section*{References}
%%%%%%%%%%%%%%%%%%%%%%%%%%%%%%%%%%%%%%%%%%%%%%%%%%%%%%%%%%%%%%%%%%

\indent $~\,\,\,\,\,$
Caramellino, L., Spizzichino, F., 1994. Dependence and aging properties of lifetimes with Schur-constant survival function. Probability in the Engineering and Informational Sciences 8, 103--111.
\par
Casta\~ner, A., Claramunt, M.M., Lef\`evre, C.; Loisel, S., 2015. Discrete Schur-constant models. Journal of Multivariate Analysis 140, 343--362.
\par
Chi, Y., Yang, J., Qi, Y., 2009. Decomposition of a Schur-constant model and its applications. Insurance: Mathematics and Economics 44, 398--408.
\par
Cox, D.R., 1962. Renewal Theory. Methuen $\&$ Co., London.
\par
Deshpande, J.V., Kochar, S.C., Singh, H., 1986. Aspects of positive ageing. Journal of Applied Probability, 23, 748--758.
\par
Gupta, R.C., 2012. Bivariate equilibrium distribution and association measure. IEEE Transactions on Reliability 61, 987--993.
\par
Lef\`evre, C., Loisel, S., 2010. Stationary-excess operator and convex stochastic orders. Insurance: Mathematics and Economics 47, 64--75.
\par
Lef\`evre, C., Loisel, S., 2013. On multiply monotone distributions, continuous or discrete , with applications. Journal of Applied Probability 50, 827--847.
\par
Lef\`evre, C., Loisel, S., Utev, S., 2017. Markov property in discrete Schur-constant models. Methodology and Computing in Applied Probability (first online).
\par
Li, S., 2011. The equilibrium distribution of counting random variables. Open Journal of Discrete Mathematics, 1, 127--135.
\par
Nair, N.U., Sankaran, P.G., 2014. Modelling lifetimes with bivariate Schur-constant equilibrium distributions from renewal theory. METRON 72, 331--349.
\par
Nelsen, R.B., 2005. Some properties of Schur-constant survival models and their copulas. Brazilian Journal of Probability and Statistics 19, 179--190.
\par
Panjer, H.H., 1981. Recursive evaluation of a family of compound distributions. ASTIN Bulletin 12, 22--26.
\par
Sunoj, S.M., 2004. Characterizations of some continuous distributions using partial moments. METRON 62, 353--362.
\par
Ta, B.Q., Van, C.P., 2017. Some properties of bivariate Schur-constant distributions. Statistics and Probability Letters 124, 69--76.
\par
Willmot, G.E., Drekic, S., Cai, J., 2005. Equilibrium compound distributions and stop-loss moments. Scandinavian Actuarial Journal 2005, 6--24.
\par
\end{document}